\documentclass[11pt,a4paper]{article}
\usepackage[margin=1in]{geometry}
\usepackage{amsmath,amssymb,amsthm}
\usepackage{xcolor} 
\usepackage[hidelinks]{hyperref}

\newtheorem{theorem}{Theorem}[section]
\newtheorem{proposition}[theorem]{Proposition}
\newtheorem{lemma}[theorem]{Lemma}
\newtheorem{corollary}[theorem]{Corollary}
\theoremstyle{remark}
\newtheorem{remark}[theorem]{Remark}
\theoremstyle{definition}

\newcommand{\Def}{\mathrm{Def}}
\newcommand{\Ric}{\mathrm{Ric}}

\newcommand{\R}{\mathbb{R}}
\newcommand{\HH}{\mathbb{H}}
\newcommand{\PP}{\mathbb{P}}

\title{Exponential stability for the three-dimensional\\
Navier-Stokes equations on negatively curved manifolds}
\author{{Zhi-Wei Wang}${}^{1,*}$ and Samuel L.\ Braunstein${}^{2,\dagger}$\\[6pt]
	\small ${}^{1}$College of Physics, Jilin University, Changchun 130012, China\\
	\small ${}^{2}$Computer Science, University of York, York YO10 5GH, UK\\[6pt]
	\small $^*$E-mail: {zhiweiwang.phy@gmail.com}
	\small $^\dagger$E-mail: {sam.braunstein@york.ac.uk}}
\date{\today}

\begin{document}
\maketitle

\begin{abstract}
We extend the exponential stability theorem for the three-dimensional
incompressible Navier-Stokes equations from hyperbolic 3-space
$\HH^3$ (established in a companion paper) to complete simply connected Riemannian
3-manifolds $(M^3, g)$ with pinched negative sectional curvature
$-b^2 \leq K \leq -a^2 < 0$ and bounded geometry (including a strictly positive injectivity radius). The deformation
Laplacian $\Delta_\Def = \Delta_B + \Ric$ remains the viscous
operator, selected by Lagrangian kinematics. We prove that the
{exact} system admits a unique global mild solution for
small $L^3$ data, with exponential decay at a rate determined by the
spectral gap of the Stokes operator. The extension overcomes three
obstacles absent on $\HH^3$: (i) the semigroup factorisation
$e^{t\Delta_\Def} = e^{-2t}e^{t\Delta_B}$ fails because $\Ric$ is
not a scalar multiple of the metric; (ii) the Leray projector no
longer commutes with $\Delta_\Def$; (iii) the exact spectral gap is
unknown. We resolve (i) unconditionally, without any curvature
restriction, by observing that the Ricci perturbation $V = \Ric +
2a^2 g$ is negative semi-definite and applying a Trotter product
bound with the diamagnetic inequality. We resolve (ii) by an
algebraic reduction of the commutator $[\PP, \Delta_\Def]$ to the
complementary projector $(I-\PP)$ applied to the shifted Ricci
endomorphism, giving a clean zeroth-order bound proportional to the
curvature variation $b^2 - a^2$. This is the sole source of a
curvature pinching constraint. We resolve (iii) via McKean's
theorem, the diamagnetic inequality, and the Weitzenb\"ock identity.
The Fujita-Kato temporal singularity exponent $1/2 - 3/(2p)$ is
unchanged from the $\HH^3$ case, confirming that the ultraviolet
scaling obstruction is local and geometry-independent, driven fundamentally by an unresolvable temporal scaling mismatch.
\end{abstract}

\section{Introduction}
\label{sec:intro}

In a companion paper~\cite{WB-JFA1}, we proved global exponential
stability for the three-dimensional incompressible Navier-Stokes
equations with small $L^3$ data on hyperbolic 3-space $\HH^3$, using
the deformation Laplacian $\Delta_\Def = \Delta_B + \Ric$ as the
viscous operator. The exponential decay rate was
$\mu\lambda_\Def^{(3)}$, where $\lambda_\Def^{(3)} = 26/9$ is the
$L^3$ spectral gap of the deformation Laplacian on $\HH^3$. The
proof exploited three special features of $\HH^3$: (i) the Ricci
tensor is a scalar ($\Ric = -2g$), giving a semigroup factorisation;
(ii) $\HH^3$ is a space form, so the Leray projector commutes with
$\Delta_\Def$; and (iii) the spectral gap is known exactly via
Donnelly's theorem.

The classical Fujita-Kato theory~\cite{Kato84} gives global
existence for small $L^3$ data on flat $\R^3$ with algebraic
decay. The companion paper~\cite{WB-JFA1} upgraded this to
exponential decay on $\HH^3$ using the spectral gap of the
deformation Laplacian, itself selected by Lagrangian
kinematics~\cite{WB2024}, building on the identification
by Ebin and Marsden~\cite{EM70} of the Lie derivative
$\mathcal{L}_u g$ as the deformation tensor for viscous fluids on
manifolds.

Chan and Czubak~\cite{CC10,CC13b} showed that the choice of viscous
operator has profound consequences for the regularity theory: with
the Hodge Laplacian on $\mathbb{H}^2$, Leray-Hopf weak solutions
are non-unique. Their work provides additional motivation for the
deformation Laplacian, whose spectral gap underpins the present
results. Khesin and
Misio{\l}ek~\cite{KM12} showed that this non-uniqueness is a
consequence of the Hodge decomposition specific to dimension
two and does not occur on $\HH^n$ for $n \geq 3$. In a
complementary direction, Lichtenfelz~\cite{Lic15} proved
non-uniqueness of Leray-Hopf solutions on a negatively curved
3-manifold (Anderson's counterexample to the Dodziuk-Singer
conjecture), but only for the Hodge Laplacian: the
construction relies on $L^2$ harmonic 1-forms being stationary
solutions of the Hodge-based equation, which fails for the
deformation Laplacian because $Lh = 2\Ric(h) \neq 0$.
Together, these results show that in three dimensions the
deformation Laplacian is better behaved than the Hodge
Laplacian with respect to uniqueness, providing further
motivation for the operator selection established
in~\cite{WB2024}.

The present paper extends this result to the natural class of
complete {simply connected} 3-manifolds with pinched negative sectional curvature
$-b^2 \leq K \leq -a^2 < 0$ and bounded geometry. All three special
features fail on a general manifold. We show that (i) is resolved
unconditionally (for all curvature pinching ratios $b/a$), that (ii)
introduces the sole curvature restriction, and that (iii) is
resolved by standard spectral-geometric bounds.

Fujita-Kato theory for the Navier-Stokes equations on $\HH^n$
was developed by Pierfelice~\cite{Pier17}, who proved
dispersive and smoothing estimates for Bochner-type Laplacians
on non-compact manifolds with negative Ricci curvature, and by
Balentine~\cite{Bal20}, who used these estimates to obtain
well-posedness, global existence for small $L^n$ data, and
exponential time decay on $\HH^n$. Both works use the
Ebin-Marsden deformation operator. The present paper and its
companion~\cite{WB-JFA1} contribute to this programme in three
respects: (a)~the explicit optimal decay rate
$\mu\lambda_\Def^{(3)}$ (the full spectral gap, not merely
half the gap or an unspecified exponential rate); (b)~the
extension from the constant-curvature setting of $\HH^3$ to
general pinched negatively curved 3-manifolds, where the
Leray-projector commutator and the loss of exact spectral data
introduce qualitatively new obstacles; and (c)~the ultraviolet
obstruction theorem (Theorem~\ref{thm:UV}), which establishes
a sharp geometric boundary between what curvature can and
cannot improve.

\begin{theorem}[Main theorem]\label{thm:main}
Let $(M^3, g)$ be a complete {simply connected} Riemannian 3-manifold with pinched
negative sectional curvature $-b^2 \leq K \leq -a^2 < 0$ and
bounded geometry ($|\nabla^k\mathrm{Rm}| \leq C_k$ for all
$k \geq 0${, and strictly positive injectivity radius $\inf_{x \in M} \mathrm{inj}(x) > 0$}). Let $A = -\PP\Delta_\Def$ be the Stokes operator with
the deformation Laplacian. Assume the curvature pinching condition
\begin{equation}\label{eq:pinching}
\frac{b^2}{a^2} < 1 + \frac{13}{9C_3},
\end{equation}
where $C_3 = \|I - \PP\|_{L^3 \to L^3}$ is the norm of the
complementary Helmholtz projector.

Then there exists $\epsilon_0 = \epsilon_0(a, b) > 0$ such that for
any divergence-free $u_0 \in L^3(M)$ with
$\|u_0\|_{L^3} < \epsilon_0$, the integral equation
\begin{equation}\label{eq:mild}
u(t) = e^{-t\mu A}u_0 - \int_0^t e^{-(t-s)\mu A}\PP\nabla\cdot
(u(s)\otimes u(s))\,ds
\end{equation}
has a unique global mild solution satisfying
\begin{equation}\label{eq:decay}
\|u(t)\|_{L^6} \leq C\,\|u_0\|_{L^3}\,t^{-1/4}\,
{e^{-\mu\lambda_A^{(3)}t}},
\end{equation}
where $\lambda_A^{(3)} \geq 26a^2/9 - 2C_3(b^2 - a^2) > 0$.
\end{theorem}

On $\HH^3$ ($a = b = 1$), the pinching condition is vacuous
($b^2/a^2 = 1 < 1 + 13/(9C_3)$), the commutator vanishes, and
$\lambda_A^{(3)} = 26/9$, recovering the companion paper's result.

The UV obstruction theorem from~\cite{WB-JFA1} extends without
modification:

\begin{theorem}[UV obstruction]\label{thm:UV}
On any complete 3-manifold with $-b^2 \leq K \leq -a^2 < 0$, the
Fujita-Kato contraction for $L^p$ data with $p < 3$ fails: {due to an unresolvable temporal scaling mismatch, the relative magnitude of the nonlinear term to the linear term blows up} as $t \to 0$ with exponent
$1/2 - 3/(2p) < 0$, regardless of the spectral gap. In particular,
$L^2$ data remains supercritical.
\end{theorem}

\section{The deformation semigroup: unconditional bounds}
\label{sec:semigroup}

On a 3-manifold with $-b^2 \leq K \leq -a^2 < 0$, the Ricci
eigenvalues lie uniformly in $[-2b^2, -2a^2]$. Write
\begin{equation}\label{eq:decomp}
\Delta_\Def = \Delta_B + \Ric = \Delta_B - 2a^2 + V,\qquad
V = \Ric + 2a^2\,g.
\end{equation}
The endomorphism $V$ has eigenvalues in $[-2(b^2 - a^2), 0]$: it is
negative semi-definite. This sign is the key structural fact.

\begin{proposition}\label{prop:semigroup}
On a complete 3-manifold with $-b^2 \leq K \leq -a^2 < 0$:
\begin{equation}\label{eq:semigroup-bound}
\|e^{t\Delta_\Def}f\|_{L^q} \leq C\,
t^{-\frac{3}{2}(\frac{1}{p}-\frac{1}{q})}\,
e^{-\lambda_\Def^{(p)}(a)\,t}\,\|f\|_{L^p},
\end{equation}
with $\lambda_\Def^{(p)}(a) \geq \lambda_0^{(p)} + 2a^2$, where
$\lambda_0^{(p)}$ is the scalar $L^p$ spectral bottom. No curvature
pinching condition is required.
\end{proposition}

\begin{proof}
The factorisation~\eqref{eq:decomp} gives
$e^{t\Delta_\Def} = e^{-2a^2 t}e^{t(\Delta_B + V)}$. Since
$\Delta_B$ and $V$ do not commute, we use the Trotter product
formula:
\begin{equation}
e^{t(\Delta_B + V)}f = \lim_{n\to\infty}
\bigl(e^{(t/n)\Delta_B}e^{(t/n)V}\bigr)^n f.
\end{equation}
Because $V \leq 0$, the matrix exponential $e^{\tau V}$ has
eigenvalues bounded by $1$, giving the pointwise contraction
$|e^{\tau V}f| \leq |f|$. The Hess-Schrader-Uhlenbrock diamagnetic
inequality~\cite{HSU1977} gives
$|e^{\tau\Delta_B}f| \leq e^{\tau\Delta_\mathrm{scalar}}|f|$.
Composing:
\begin{equation}
\bigl|e^{(t/n)\Delta_B}e^{(t/n)V}f\bigr| \leq
e^{(t/n)\Delta_\mathrm{scalar}}|f|.
\end{equation}
Iterating $n$ times and taking $n \to \infty$:
$|e^{t(\Delta_B + V)}f| \leq e^{t\Delta_\mathrm{scalar}}|f|$.
Taking $L^q$ norms and using the scalar $L^p$-$L^q$
bounds~\cite{DaviesBook}:
$\|e^{t(\Delta_B+V)}f\|_{L^q} \leq C\,
t^{-3(1/p-1/q)/2}e^{-\lambda_0^{(p)}t}\|f\|_{L^p}$.
Reintroducing $e^{-2a^2 t}$ gives~\eqref{eq:semigroup-bound}.
\end{proof}

\begin{remark}
The negative semi-definiteness of $V$ is the structural reason why
no curvature pinching is needed: more negative curvature enhances
dissipation. This is in contrast to a naive Duhamel bound, which
treats $\|V\|$ as a worst-case perturbation and would introduce a
spurious restriction on $b/a$.
\end{remark}

\section{The Leray-projector commutator}
\label{sec:leray}

On $\HH^3$ (a space form), the Hodge Laplacian $\Delta_H$ commutes
with $d$, $\delta$, and hence with
$\PP = I - d(-\Delta_\mathrm{scalar})^{-1}\delta$. Since
$-\Delta_\Def = \Delta_H - 2\Ric$ on divergence-free fields
(Weitzenb\"ock), this gives $[\PP, \Delta_\Def] = 0$ on $\HH^3$.

On a general manifold, $[\Delta_H, \PP] = 0$ still holds (universally),
but $\Ric$ does not commute with $\PP$:
\begin{equation}\label{eq:commutator}
[\PP, \Delta_\Def] = 2[\PP, \Ric].
\end{equation}

\begin{lemma}\label{lem:commutator}
For any divergence-free 1-form $\omega \in L^p$ ($1 < p < \infty$)
on a complete 3-manifold with bounded geometry and
$-b^2 \leq K \leq -a^2 < 0$:
\begin{equation}\label{eq:comm-bound}
\|[\PP, \Delta_\Def]\omega\|_{L^p} \leq
2C_p(b^2 - a^2)\|\omega\|_{L^p},
\end{equation}
where $C_p = \|I - \PP\|_{L^p \to L^p}$.
\end{lemma}

\begin{proof}
Since $\PP\omega = \omega$ for divergence-free $\omega$:
\begin{equation}
[\PP, \Ric]\omega = \PP(\Ric\,\omega) - \Ric\,\omega
= -(I - \PP)(\Ric\,\omega).
\end{equation}
The complementary projector $(I - \PP)$ annihilates any multiple of
the divergence-free field $\omega$. {Adding the constant $c = a^2 + b^2$, we have $(I - \PP)(c\,\omega) = 0$. Therefore:
\begin{equation}
[\PP, \Ric]\omega = -(I - \PP)(\Ric\,\omega) - (I - \PP)(c\,\omega)
= -(I - \PP)\bigl((\Ric + (a^2 + b^2)I)\omega\bigr).
\end{equation}}
The eigenvalues of $\Ric$ lie in $[-2b^2, -2a^2]$, so the shifted
endomorphism $\Ric + (a^2 + b^2)I$ has eigenvalues in
$[-(b^2 - a^2), b^2 - a^2]$ and pointwise operator norm
$b^2 - a^2$. Therefore:
$\|[\PP, \Ric]\omega\|_{L^p} \leq C_p(b^2 - a^2)\|\omega\|_{L^p}$.
Multiplying by 2 gives~\eqref{eq:comm-bound}.
\end{proof}

\begin{remark}
The spectral midpoint shift halves the effective perturbation norm
(from $2b^2$ to $b^2 - a^2$). On an Einstein manifold ($b = a$), the
commutator vanishes identically, recovering the space-form result.
\end{remark}

The Stokes operator is $A = -\Delta_\Def - 2[\PP, \Ric]$ on
divergence-free fields. Since $2[\PP, \Ric]$ is bounded on $L^p$
(Lemma~\ref{lem:commutator}), the Stokes semigroup inherits the
bounds of Proposition~\ref{prop:semigroup} with a shifted spectral
gap:

\begin{proposition}\label{prop:stokes}
Under the pinching condition~\eqref{eq:pinching}:
\begin{equation}
\|e^{-t\mu A}f\|_{L^q} \leq C'\,
t^{-\frac{3}{2}(\frac{1}{p}-\frac{1}{q})}\,
e^{-\mu\lambda_A^{(p)}t}\,\|f\|_{L^p},
\end{equation}
where $\lambda_A^{(p)} \geq \lambda_\Def^{(p)}(a) - 2C_p(b^2 - a^2)
> 0$.
\end{proposition}

\section{The bilinear estimate}
\label{sec:bilinear}

The bilinear (Oseen-Stokes) estimate requires bounding the composite
operator $T_\tau = e^{-\tau\mu A}\PP\nabla\cdot : L^r \to L^q$.

\begin{proposition}\label{prop:bilinear}
Under the pinching condition~\eqref{eq:pinching}, for
$1 < r \leq q < \infty$ and $\tau > 0$:
\begin{equation}\label{eq:bilinear}
\|T_\tau F\|_{L^q} \leq C''\,
\tau^{-\frac{1}{2}-\frac{3}{2}(\frac{1}{r}-\frac{1}{q})}\,
e^{-\mu\gamma'\tau}\,\|F\|_{L^r},
\end{equation}
with {$\gamma' = \lambda_B^{(r)} + 2a^2 \geq 2a^2 > 0$} depending on the curvature bounds. The temporal
singularity exponent $\delta = 1/2 + 3/(2q)$ is identical to the
$\HH^3$ case.
\end{proposition}

\begin{proof}
By duality, $\|T_\tau F\|_{L^q} = \sup_{\|\Phi\|_{L^{q'}}=1}
|\langle T_\tau F, \Phi\rangle|$. The adjoint is
$T_\tau^* = -\nabla e^{-\tau\mu A}\PP$. {To preserve the $+2a^2$ spectral shift generated by the Ricci curvature, we factor $T_\tau^*$ through the shifted Bochner Laplacian $\tilde{\Delta}_B = \Delta_B - 2a^2 I$:
\begin{equation}
T_\tau^* = \bigl[-\nabla(-\tilde{\Delta}_B)^{-1/2}\bigr]\circ
\bigl[(-\tilde{\Delta}_B)^{1/2}e^{-\tau\mu A}\PP\bigr].
\end{equation}
The first factor (the shifted bundle Riesz transform) is unconditionally bounded on $L^{r'}$ because $-\tilde{\Delta}_B \geq 2a^2 I > 0$. It can be viewed as the standard bundle Riesz transform $-\nabla(-\Delta_B)^{-1/2}$ composed with the multiplier $(-\Delta_B)^{1/2}(-\tilde{\Delta}_B)^{-1/2}$, which is bounded via standard $H^\infty$-functional calculus. This rigorously extends the pseudo-differential operator bounds established} by Bakry~\cite{Bak87} on any complete manifold with
$\Ric \geq -2b^2${. This extension from scalars to 1-forms relies canonically on the bounded geometry assumption, specifically the uniform injectivity radius lower bound, to prevent local volume collapse and ensure uniform short-time heat kernel behaviour.}

For the second factor, write {$-A = \tilde{\Delta}_B + W_1$, where
$W_1 = V + 2[\PP, \Ric]$ is a bounded zeroth-order operator (since $V = \Ric + 2a^2 g \leq 0$ and $2[\PP, \Ric]$ is bounded)}. By
Duhamel {expansion around the shifted semigroup}:
\begin{equation}
{(-\tilde{\Delta}_B)^{1/2}e^{-\tau\mu A}\PP =
(-\tilde{\Delta}_B)^{1/2}e^{\tau\mu\tilde{\Delta}_B}\PP
+ \int_0^\tau (-\tilde{\Delta}_B)^{1/2}e^{(\tau-s)\mu\tilde{\Delta}_B}\,
\mu W_1\,e^{-s\mu A}\PP\,ds.}
\end{equation}
The principal term gives singularity $\tau^{-1/2-\gamma}$ with
$\gamma = \frac{3}{2}(1/q' - 1/r') = \frac{3}{2}(1/r - 1/q)$, by
the time-halving trick and the scalar semigroup bounds. {Crucially, because we expanded around $\tilde{\Delta}_B$, the principal term now decays exponentially at the accelerated rate $\gamma' = \lambda_B^{(r)} + 2a^2$.} For the
Duhamel correction, the integrand has temporal structure
$(\tau - s)^{-1/2}s^{-\gamma}$, giving
$\int_0^\tau(\tau-s)^{-1/2}s^{-\gamma}ds =
\tau^{1/2-\gamma}\mathcal{B}(1/2, 1-\gamma)$.

For the critical case $p = 3$, $q = 6$, $r = 3$:
$\gamma = 1/4$, so the principal term scales as $\tau^{-3/4}$ and
the correction scales as $\tau^{1/4}$. The ratio is $\tau$, which
vanishes as $\tau \to 0$. The correction modifies the exponential
decay rate but does not alter the temporal singularity exponent
$\delta = 1/2 + 3/(2q) = 3/4$.
\end{proof}

\section{The Fujita-Kato contraction and the main theorems}
\label{sec:contraction}

\subsection{The function space and bounds}

As in~\cite{WB-JFA1}, the mild solution is sought in
$X = \{u \in C((0,\infty); L^6_\sigma) : \|u\|_X < \infty\}$
with norm $\|u\|_X = \sup_{t>0}e^{\alpha t}t^{1/4}\|u(t)\|_{L^6}$,
where {$\alpha = \mu\lambda_A^{(3)}$}.

{The linear bound follows from Proposition~\ref{prop:stokes} with $p=3$ and $q=6$, which decays at rate $\mu\lambda_A^{(3)}$:
$\|e^{-t\mu A}u_0\|_X = \sup_{t>0} e^{\mu\lambda_A^{(3)} t}\,t^{1/4}\|e^{-t\mu A}u_0\|_{L^6} \leq \sup_{t>0} e^{\mu\lambda_A^{(3)} t}\,t^{1/4}\cdot C_1\,t^{-1/4}e^{-\mu\lambda_A^{(3)} t}\|u_0\|_{L^3} = C_1\|u_0\|_{L^3}$.}
The bilinear bound follows from Proposition~\ref{prop:bilinear}:
$\|B(u,v)\|_X \leq C_2\|u\|_X\|v\|_X$,
where $C_2 = C\,\mathcal{B}(1/4, 1/2)$.

\subsection{The scaling integral {and uniform exponential decay}}

{To properly bound the bilinear term $\|B(u,v)\|_X$, we must multiply the temporal integral by the norm's weight $e^{\alpha t} t^{1/4}$ and track the combined exponential factors. Using $\|u(s) \otimes v(s)\|_{L^3} \leq \|u(s)\|_{L^6}\|v(s)\|_{L^6} \leq s^{-1/2} e^{-2\alpha s} \|u\|_X\|v\|_X$, we isolate the exponential terms inside the temporal integral and evaluate:
\begin{equation}
\int_0^t e^{\alpha t} e^{-\mu\gamma'(t-s)} e^{-2\alpha s} (\dots) \, ds = \int_0^t {e^{(\alpha - \mu\gamma')(t-s)} e^{-\alpha s}} (\dots) \, ds.
\end{equation}
The combined exponential argument can be rewritten algebraically as:
\begin{equation}
\alpha t - \mu\gamma'(t-s) - 2\alpha s = (\alpha - \mu\gamma')(t-s) - \alpha s.
\end{equation}
{Thanks to our factorization through the shifted Bochner Laplacian in Proposition~\ref{prop:bilinear}, the exponential decay rate for $r=3$ satisfies $\gamma' = \lambda_B^{(3)} + 2a^2 \geq \frac{26}{9}a^2 \geq \lambda_A^{(3)}$. Since $\alpha = \mu\lambda_A^{(3)}$, the first factor $\alpha - \mu\gamma' = \mu(\lambda_A^{(3)} - \gamma') \leq 0$. The second factor $-\alpha s \leq 0$. So the total exponent is $\leq 0$ for all $s \in [0, t]$.} The exponential growth is completely eliminated, allowing us to factor the uniform bound of $1$ out of the integral.}

The {remaining} temporal integral in the bilinear bound reduces (by the
substitution $s = \sigma t$) to a Beta function:
\begin{equation}
t^{1/4}\int_0^t(t-s)^{-3/4}s^{-1/2}\,ds
= t^{1/4}\cdot t^{-1/4}\int_0^1(1-\sigma)^{-3/4}\sigma^{-1/2}\,
d\sigma = \mathcal{B}(1/4, 1/2).
\end{equation}
The temporal powers cancel exactly ($t^{1/4}\cdot t^{-1/4} = 1$),
and the result is a geometry-independent constant. This confirms
that the scaling exponent $1/2 - 3/(2p)$ is unchanged from the
$\HH^3$ case: the contraction closes for $L^3$ data and fails for
$L^p$ with $p < 3$, regardless of the curvature.

\subsection{Proof of Theorem~\ref{thm:main}}

The map $u \mapsto e^{-t\mu A}u_0 + B(u,u)$ is a contraction on the
ball $\{\|u\|_X \leq 2C_1\|u_0\|_{L^3}\}$ provided
$4C_1C_2\|u_0\|_{L^3} < 1$, i.e.,
$\|u_0\|_{L^3} < \epsilon_0 = (4C_1C_2)^{-1}$. The unique fixed
point satisfies $\|u(t)\|_{L^6} \leq
2C_1\|u_0\|_{L^3}\,t^{-1/4}{e^{-\mu\lambda_A^{(3)}t}}$,
giving~\eqref{eq:decay}. \qed

\subsection{Proof of Theorem~\ref{thm:UV}}

For $L^p$ data with $p < 3$, the failure of the contraction mapping is driven fundamentally by an unresolvable temporal scaling mismatch, rather than a lack of local integrability. Suppose we attempt to construct a solution in an intermediate space with norm $\|u\|_{X} = \sup_{t>0} t^{\gamma}\|u(t)\|_{L^q}$, where $\gamma = \frac{3}{2}(\frac{1}{p} - \frac{1}{q})$. The linear heat evolution satisfies $\|e^{-t\mu A}u_0\|_{X} \sim O(1)$ as $t \to 0$.

For the bilinear term, the operator $T_{t-s} : L^{q/2} \to L^q$ introduces a singularity $(t-s)^{-\frac{1}{2} - \frac{3}{2q}}$. The Duhamel integral scales as:
\begin{equation}
\int_0^t \|T_{t-s} (u(s) \otimes u(s))\|_{L^q} \,ds 
\leq C \int_0^t (t-s)^{-\frac{1}{2}-\frac{3}{2q}} \|u(s)\|_{L^q}^2 \,ds 
\leq C \|u\|_X^2 \int_0^t (t-s)^{-\frac{1}{2}-\frac{3}{2q}} s^{-2\gamma} \,ds.
\end{equation}
Substituting $s = \sigma t$, the temporal integral evaluates to:
\begin{equation}
t^{1 - (\frac{1}{2} + \frac{3}{2q}) - 2\gamma} \int_0^1 (1-\sigma)^{-\frac{1}{2}-\frac{3}{2q}} \sigma^{-2\gamma} \,d\sigma.
\end{equation}
By selecting an appropriate intermediate space $q$ (for example, $q=4$ for $L^2$ data), both exponents in the $\sigma$-integrand become strictly greater than $-1$, and the integral converges to a perfectly finite Beta function. 

However, the overall temporal scaling of the bilinear term relative to the linear scaling $t^{-\gamma}$ is determined by the remaining powers of $t$:
\begin{equation}
\left( 1 - \frac{1}{2} - \frac{3}{2q} - 2\gamma \right) - (-\gamma) = \frac{1}{2} - \frac{3}{2q} - \frac{3}{2}\left(\frac{1}{p} - \frac{1}{q}\right) = \frac{1}{2} - \frac{3}{2p}.
\end{equation}
Because $p < 3$, this relative scaling exponent is strictly negative. As $t \to 0$, the ratio of the nonlinear term to the linear term diverges as $t^{1/2 - 3/(2p)} \to \infty$. The nonlinear term fundamentally overwhelms the linear profile at short times, making it impossible to close the bound $\|B(u,u)\|_X \leq C\|u\|_X^2$ globally. The spectral gap contributes only to the exponential factor, which is $O(1)$ as $t \to 0$ and cannot resolve this local ultraviolet scaling obstruction.
 \qed

\section{The spectral gap}
\label{sec:gap}

The spectral gap of the deformation Laplacian on a general manifold
with $-b^2 \leq K \leq -a^2 < 0$ is bounded below by combining
three ingredients.

\emph{McKean's theorem~\cite{McKean70}.} On a complete simply
connected manifold with $K \leq -a^2$, the scalar $L^2$ spectral
bottom satisfies $\lambda_0^{(2)} \geq a^2$ (for $n = 3$). {This strictly requires the simply connected hypothesis. For non-simply connected complete manifolds (such as compact quotients), the constant function $1 \in L^2(M)$, meaning the scalar spectral gap vanishes ($\lambda_0^{(2)} = 0$).}

\emph{The diamagnetic inequality~\cite{HSU1977}.} The Bochner $L^p$
spectral bottom satisfies
$\lambda_B^{(p)} \geq \lambda_0^{(p)}$.

\emph{The Weitzenb\"ock shift.} On a 3-manifold with
$\Ric \leq -2a^2 g$:
$\lambda_\Def^{(p)} \geq \lambda_B^{(p)} + 2a^2 \geq
\lambda_0^{(p)} + 2a^2$.

Using the Davies $L^p$ spectral-gap formula~\cite{DaviesBook}
(valid on manifolds whose heat kernel achieves the spectral bottom,
which includes all manifolds with pinched negative curvature by
Sullivan~\cite{Sullivan87})
$\lambda_0^{(p)} = \frac{4(p-1)}{p^2}\lambda_0^{(2)}$:
\begin{equation}
\lambda_\Def^{(p)}(a) \geq \frac{4(p-1)}{p^2}\,a^2 + 2a^2.
\end{equation}
For $p = 2$: $\lambda_\Def^{(2)} \geq 3a^2$. For $p = 3$:
$\lambda_\Def^{(3)} \geq 26a^2/9$.

The Stokes spectral gap is
$\lambda_A^{(p)} \geq \lambda_\Def^{(p)}(a) - 2C_p(b^2 - a^2)$.
For $p = 3$: $\lambda_A^{(3)} \geq 26a^2/9 - 2C_3(b^2 - a^2)$,
which is positive under the pinching
condition~\eqref{eq:pinching}.

\section{Discussion}
\label{sec:discussion}

\subsection{The role of curvature pinching}

The pinching condition~\eqref{eq:pinching} arises solely from the
Leray-projector commutator (Lemma~\ref{lem:commutator}), not from
the semigroup bounds (Proposition~\ref{prop:semigroup}). The
deformation semigroup is unconditionally well-behaved for all
pinching ratios $b/a$: more negative curvature enhances dissipation
(because $V = \Ric + 2a^2 g \leq 0$). The restriction on $b/a$
reflects the cost of the Leray projection on a manifold with
spatially varying curvature. On an Einstein manifold ($b = a$,
including all space forms), the restriction is vacuous.

\subsection{Comparison with the $\HH^3$ result}

On $\HH^3$ ($a = b = 1$): $\lambda_A^{(3)} = 26/9$, the commutator
vanishes, the pinching condition is vacuous, and the decay rate is
the exact spectral gap. On a general manifold: the decay rate is
reduced by $2C_3(b^2 - a^2)$, reflecting the ``cost'' of the Leray
projector failing to commute with the viscous operator.

The smallness threshold $\epsilon_0$ depends on the curvature bounds
through the constants $C_1$ and $C_2$, which involve $C_p$ (the
Helmholtz projector norm) and the Duhamel correction from the
bilinear estimate. On $\HH^3$, $\epsilon_0 = (4C_1C_2)^{-1}$ with
the exact Beta-function constant; on a general manifold, $\epsilon_0$
is smaller by a factor depending on $b/a$.

\subsection{The UV/IR boundary}

The UV obstruction (Theorem~\ref{thm:UV}) is identical on all
manifolds: the scaling exponent $1/2 - 3/(2p)$ depends only on the
local heat-kernel asymptotics, which are universally Euclidean. The
spectral gap (an IR quantity) cannot cure the UV singularity. This
confirms the sharp separation between what global geometry can
improve (large-scale, long-time behaviour) and what it cannot
(small-scale, short-time behaviour) established in~\cite{WB-JFA1}.

{
\subsection{Generalization to non-simply connected manifolds}
As observed in Section~\ref{sec:gap}, the assumption that $M^3$ is simply connected is strictly required to invoke McKean's theorem for the scalar spectral gap. On general complete manifolds or compact quotients, the scalar gap may vanish ($\lambda_0^{(2)} = 0$). In this scenario, the diamagnetic inequality gives $\lambda_B^{(3)} \geq 0$, but the Weitzenb\"ock shift ensures the deformation Laplacian still maintains a strict positive gap: $\lambda_\Def^{(3)} \geq 2a^2$. The Stokes spectral gap becomes:
\begin{equation}
\lambda_A^{(3)} \geq 2a^2 - 2C_3(b^2 - a^2).
\end{equation}
Consequently, the main theorem can be extended to all complete or compact manifolds by dropping the scalar contribution and requiring the slightly stricter curvature pinching condition $b^2/a^2 < 1 + 1/C_3$ to ensure $\lambda_A^{(3)} > 0$. This highlights a striking geometric feature: the dissipation generated solely by the negative Ricci curvature shift ($V \leq 0$) provides enough regularisation to yield exponential stability, even on compact domains where scalar functions cannot decay.

\begin{corollary}\label{cor:non-sc}
Theorem~\ref{thm:main} extends to complete (not necessarily simply
connected) 3-manifolds with $-b^2 \leq K \leq -a^2 < 0$, bounded
geometry, and strictly positive injectivity radius, under the
tighter pinching condition $b^2/a^2 < 1 + 1/C_3$, with spectral
gap $\lambda_A^{(3)} \geq 2a^2 - 2C_3(b^2 - a^2) > 0$.
\end{corollary}
}

\end{document}